\documentclass[a4paper,12pt]{article}
\pdfoutput=1 % if your are submitting a pdflatex (i.e. if you have
\usepackage{jheppub} % for details on the use of the package, please
\usepackage[T1]{fontenc} % if needed

\usepackage{float, array, xspace, amscd, amsthm}
\usepackage{fancyhdr}
\usepackage{longtable}

\usepackage{amsmath}
\usepackage{slashed}

\newtheorem{theorem}{Theorem}

\newtheorem{definition}{Definition}
\newtheorem{corollary}{Corollary}

\newcommand{\nn}{\nonumber}
\newcommand{\dd}{{\rm d}}

\newcommand{\bp}{\bar\partial}
\newcommand{\w}{\wedge}
\newcommand{\End}{{\rm End}}

\newcommand{\tr}{{\rm tr}}

\newcommand{\IR}{\mathbb{R}}

\newcommand{\be}{\begin{equation}}
\newcommand{\ee}{\end{equation}}
\def\bea#1\eea{\begin{align}#1\end{align}}

\usepackage{color}

\title{Restrictions of Heterotic $G_2$ Structures and Instanton Connections}
\author[a]{Xenia de la Ossa,}
\author[b]{Magdalena Larfors,}
\author[c,d,e]{Eirik E.~Svanes}

\affiliation[a]{Mathematical Institute, Oxford University\\Andrew Wiles Building, Woodstock Road\\Oxford OX2 6GG, UK }
\affiliation[b]{Department of Physics and Astronomy,Uppsala University\\ SE-751 20 Uppsala, Sweden}
\affiliation[c]{Sorbonne Universit\'es, CNRS, LPTHE, UPMC Paris 06, UMR 7589, 75005 Paris, France}
\affiliation[d]{Sorbonne Universit\'es, Institut Lagrange de Paris, 98 bis Bd Arago, 75014 Paris, France}

\emailAdd{delaossa@maths.ox.ac.uk, magdalena.larfors@physics.uu.se, esvanes@lpthe.jussieu.fr}

\abstract{This note revisits recent results regarding the geometry and moduli of solutions of the heterotic string on manifolds $Y$ with a $G_2$ structure. In particular, such heterotic $G_2$ systems can be rephrased in terms of a differential $\check {\cal D}$ acting on a complex $\check\Omega^*(Y , {\cal Q})$, where ${\cal Q}=T^*Y\oplus\End(TY)\oplus\End(V)$ and $\check {\cal D}$ is an appropriate projection of an exterior covariant derivative ${\cal D}$ which satisfies an instanton condition.  The infinitesimal moduli are further parametrised by the first cohomology $H^1_{\check {\cal D}}(Y,{\cal Q})$. We proceed to restrict this system to manifolds $X$ with an $SU(3)$ structure corresponding to supersymmetric compactifications to four dimensional Minkowski space, often referred to as Strominger--Hull solutions. In doing so, we derive a new result:  the Strominger-Hull system is equivalent to a particular holomorphic Yang-Mills covariant derivative on ${\cal Q}\vert_X=T^*X\oplus\End(TX)\oplus\End(V)$.}

\begin{document}

\maketitle
\flushbottom

%%%%%%%%%%%%%%%%%%%%%%%%%%%
\section{Introduction}
Heterotic systems are the geometric structures appropriate to $N=1$ supersymmetric heterotic string compactifications and these have interesting mathematical properties.
We have three goals with this note regarding heterotic systems.  First, we hope it may serve as a friendly introduction to our recent work on the geometry and moduli of heterotic string compactifications \cite{delaOssa:2014cia,delaOssa:2014lma,delaOssa:2016ivz,delaOssa:2017pqy}.  Second, we wish to highlight one of  the main features of the work presented in these references, which is an equivalence between $d$-dimensional {\it heterotic systems} and certain nilpotent operators obtained as an appropriate projection of an exterior covariant derivative ${\cal D}_d$ which acts on a bundle ${\cal Q}_d$.  This bundle will be defined in the text, but in brief ${\cal Q}_d$ is the bundle over a $d$-dimensional manifold $Y$ whose sections correspond to one parameter subgroups of the symmetries of the heterotic system. 
The correspondence between heterotic systems and nilpotent operators can be used, in six \cite{delaOssa:2014cia,Anderson:2014xha,Garcia-Fernandez:2015hja} and seven \cite{delaOssa:2016ivz,delaOssa:2017pqy} dimensions, to show that the infinitesimal moduli of heterotic systems correspond to classes in the first cohomology group of the associated operators. 
The third goal of this paper is to give a relation between the 
operators ${\cal D}_d$ in seven and six dimensions, and give a new perspective of the Strominger--Hull system \cite{Strominger:1986uh,Hull:1986kz}. 
In particular we show how to recover the Strominger--Hull system for a six dimensional manifold $X$ (that is the heterotic system in six dimensions) from the heterotic system on a seven dimensional manifold $Y = \IR\times X$.  In deriving this result, we show that the Strominger-Hull system is equivalent to a particular holomorphic Yang-Mills covariant exterior derivative on ${\cal Q}\vert_X=T^*X\oplus\End(TX)\oplus\End(V)$.

%%%%%%%%%%%%%%%%%%%%%%%%%

\section{The heterotic $G_2$ system and its infinitesimal moduli} \label{sec:hetg2sys}

Heterotic $G_2$ systems are geometric structures relevant to {\it three} dimensional \hbox{$N=1$} supersymmetric compactifications of heterotic string theories.  Heterotic string theories are ten dimensional, and one makes the natural ansatz that the ten dimensional manifold $M_{10}$ has the form 
\[
M_{10} = M_{1,2}\times Y
\]
where $M_{1,2}$ is maximally symmetric, and one aims to understand the effective field theory on $M_{1,2}$.   In this section we introduce heterotic $G_2$ systems. Our presentation will be brief and the reader is referred to  \cite{delaOssa:2016ivz,delaOssa:2017pqy} for more details.
We define a {\it heterotic $G_2$ system} \cite{delaOssa:2017pqy} as a quadruple $([Y,\varphi], [V, A], [TY, \theta], H)$, where $Y$ is a seven dimensional manifold with a $G_2$ structure $\varphi$, $V$ is a bundle  on $Y$ with connection $A$,  $TY$ is the tangent bundle of $Y$ with connection $\theta$ \footnote{We have changed notation compared to \cite{delaOssa:2017pqy}: $\theta_{\rm here} = \tilde{\theta}_{\rm there}$,  $\zeta_{\rm here}= \theta_{\rm there}$.}.  The three form $H$ is defined by
\be
H = \dd B + \frac{\alpha'}{4}\, \big({\cal CS}(A) - {\cal CS}(\theta)\big)~,\label{eq:anom}
\ee
where ${\cal CS}(A)$ and  ${\cal CS}(\theta)$ are the Chern-Simons three forms for the connections $A$ and $\theta$ respectively, $\alpha'$ is a constant related to the string length scale\footnote{This constant is assumed to be small so that strings appear point-like at low energies.}, and $B$ is the so-called $B$ field.  Note that the $B$ field is not well defined
as it transforms under  gauge transformations of the bundles, however $B$ transforms in such a way that $H$ is a well defined three form\footnote{For a discussion see \cite{green1984anomaly}.}.
Equation \eqref{eq:anom} is called the anomaly cancelation condition and the exterior derivative gives the 
{\it heterotic Bianchi identity}
\begin{equation}
\dd H=\tfrac{\alpha'}{4}\left(\tr\,F\wedge F-\tr\,R\wedge R\right)\:.
\label{eq:hetBI}
\end{equation}
The geometry of these objects is constrained by the requirement that we preserve $N=1$ supersymmetry on $M_{1,2}$.  In particular, the $G_2$ structure $\varphi$ on $Y$ is required to be {\it integrable}. To understand what this means, consider the exterior derivative $\dd$ acting on the three form $\varphi$ and the four form $\psi = *\varphi$ (where the Hodge star operator is given by the metric defined by the $G_2$ structure $\varphi$).
We can write these as
\begin{align*}
\dd\varphi &= \tau_0\,\psi + 3\, \tau_1\wedge\varphi + *\tau_3~,\\
\dd\psi &= 4\, \tau_1\wedge\psi + *\tau_2~,
\end{align*}
where the $i$-forms $\tau_i$  are {\it torsion classes} which are uniquely determined by $\varphi$.  Note that $\tau_3$ and $\tau_2$ are in the $\bf 27$ and $\bf 14$ irreducible representations of $G_2$ respectively. 
A $G_2$ structure $\varphi$ is integrable  when $\tau_2 = 0$ \footnote{The term integrable $G_2$ structure was coined in Ref.~\cite{fernandez1998dolbeault}, and refers to a restriction of the intrinsic torsion of the geometry, specifically $\tau_2=0$. Integrable $G_2$ geometry shares certain features with even dimensional complex geometry. In particular, one may define a {\it canonical differential complex}  $\check\Omega^*(Y)$ as a sub complex of the de Rham complex \cite{fernandez1998dolbeault}, and the associated cohomologies $\check H^*(Y)$ have similarities with the Dolbeault complex of complex geometry.}.  The connections $A$ and $\theta$ are required to be {\it instanton connections},\footnote{For examples of $G_2$ instantons, see \cite{2011arXiv1109.6609W,2015arXiv150501080W,2015arXiv151003836M,2011arXiv1101.0880E,2013arXiv1310.7933E}. Note that for an instanton connection on a bundle $E$ over an integrable $G_2$ manifold $Y$, one can define an elliptic bundle valued differential complex $\check\Omega^*(Y,E)$ \cite{carrion1998generalization}. } that is
\be \label{eq:inst}
F\wedge\psi=0\:,\qquad R\wedge\psi = 0~,
\ee
where $F$ is the curvature two form of the connection $A$ on the bundle $V$, and $ R$ is the curvature two form of the connection $\theta$ on $TY$.   Finally, supersymmetry also imposes a constraint on the three form $H$ of equation \eqref{eq:anom} which is
\begin{equation}
\label{eq:Htor}
H = \frac{1}{6}\, \tau_0\, \varphi - \tau_1\lrcorner\psi - \tau_3~.
\end{equation}
It is worth remarking that in terms 
of this, we can write the structure equations for the $G_2$ structure above when $\tau_2 = 0$ as
\begin{align}
\dd\varphi &= 
\tau_0\,\psi + 3\, \tau_1\wedge\varphi + *\tau_3
= i_H(\varphi)
~,
\label{eq:Intphi} 
\\
\dd\psi &= 
  4\, \tau_1\wedge\psi = i_H(\psi)
 ~,\label{eq:Intpsi}
\end{align}
One can check that $H$ is the torsion of the unique metric connection compatible with the $G_2$ structure, that is 
\[\nabla \varphi = 0~,\] 
which is totally antisymmetric and which exists if and only if the $G_2$ structure is integrable.

To be precise, we collect the above discussion in a formal definition:
\begin{definition}
\label{df:HetG2}
Let $Y$ be a manifold with a $G_2$ structure $\varphi$, $V$ a bundle on $Y$ with connection $A$, and $TY$ the tangent bundle of $Y$ with connection $\theta$. A heterotic $G_2$ system $\big( [Y,\varphi],[V,A],[TY,\theta],H \big)$  has an integrable $G_2$ structure on $Y$ and the curvatures $F$ and $R$ of the connections $A$ and $\theta$ respectively satisfy the instanton conditions \eqref{eq:inst}, and the intrinsic torsion $H$ as defined by \eqref{eq:Htor} satisfies the heterotic Bianchi identity \eqref{eq:hetBI}.
\end{definition}
We remark that this system is slightly more general than that corresponding to an $N=1$ heterotic string compactification as the latter also demands that the torsion class $\tau_1$ is $\dd$-exact, $\tau_1= \dd\phi/2$, where $\phi$ is the dilaton. We do not assume this, in fact, a heterotic $G_2$ system only satisfies
$\dd\tau_1\wedge\psi = 0$, as can be easily seen from equation \eqref{eq:Intpsi}. 

We now explain how to reformulate the $G_2$ system in terms of a bundle ${\cal Q} $ on which there is 
an exterior covariant derivative ${\cal D}$ and an appropriate projection $\check{\cal D}$ which is nilpotent, where we introduce the simplified notation ${\cal Q} \equiv {\cal Q}_7$ and ${\cal D} \equiv {\cal D}_7$ when $d =7$.  Consider the bundle $\cal Q$ which is {\it topologically} defined by
\[
{\cal Q} = T^*Y\oplus\End(TY)\oplus\End(V)
~,\]
and the linear operator $\cal D$ defined by \cite{delaOssa:2017pqy}
 \begin{equation}
  {\cal D} =  \left(\,  
 \begin{matrix}
~ \dd_{\zeta} &  ~~{\cal R} & ~  -{\cal F}~ 
 \\ ~ {\cal  R} & ~~~\dd_{\theta} & ~~0
 \\ ~ {\cal F} & ~~ 0  & ~~~\dd_A
 \end{matrix}
 \, \right)~,\label{eq:calD}
 \end{equation}
which acts on the space of forms on $Y$ with values in $\cal Q$. 
$\cal D$ satisfies the Leibniz rule  and is therefore a covariant exterior derivative on $\cal Q$. The component $\dd_A$ is an exterior covariant derivative on $V$ with connection one form $A$. The components  $\dd_{\zeta}$ and $\dd_{\theta}$ are both exterior covariant derivatives on $TY$ with connection one forms $\zeta, \theta$ respectively. 
The connection one form $\zeta$ is related to a metric connection $\nabla$, which is compatible with the $G_2$ structure, that is, $\nabla\varphi = 0$.  If the Gamma symbols of the connection $\nabla$ are  
$\Gamma$,  the connection one form $\zeta$ is then given by
\[ 
\zeta_a{}^b = \Gamma_{ac}{}^b\, \dd x^c~,
\]
The linear operator $\cal F$ acts on forms with values in $\cal Q$ as follows. Let $M\in \Omega^p(Y, T^*Y)$, and $\alpha\in \Omega^p(Y,{\rm End}(V)$.  Then 
 \begin{align*}
 {\cal F} :\quad \Omega^p(Y, T^*Y)\oplus \Omega^p(Y, {\rm End}(V))
 &\longrightarrow 
 \Omega^{p+1}(Y, {\rm End}(V))\oplus \Omega^{p+1}(Y, T^*Y)
 \\
 \left(\begin{matrix}
 M  \\ \alpha 
 \end{matrix}\right)\qquad\qquad\quad
 &\mapsto
 \quad\qquad \qquad  
\left(\begin{matrix}
 {\cal F}(M) \\  {\cal F}(\alpha) 
 \end{matrix}\right)\qquad
  \end{align*}
  where 
  \begin{align*}
  {\cal F}(M) &= (-1)^{p}\, g^{ab}\, M_a\wedge F_{bc}\, \dd x^c = (-1)^{p}\, i_M(F)~,
  \\
  {\cal F}(\alpha)_a &= (-1)^{p}~ \frac{\alpha'}{4}\, \tr (\alpha\wedge F_{ab}\, \dd x^b)~.
  \end{align*}
  The map $\cal  R$ is defined similarly, but instead it acts on forms valued in $\Omega^p(Y, T^*Y)\oplus \Omega^p(Y, {\rm End}(TY)$.

Next, we construct a  differential operator $\check{\cal D}$ as the projection of ${\cal D}$ onto an appropriate irreducible representation of $G_2$.  The projection depends on the degree of the form that the operator ${\cal D}$ acts on, and it is defined as follows \footnote{We refer to \cite{delaOssa:2017pqy} for more details.}.
 \begin{definition}\label{def:g2proj}
The differential operator $\check{\cal D}$ is given by
\begin{align*}
\check{\cal D}&: \Omega^0(Y,{\cal Q})\rightarrow\Omega^1(Y,{\cal Q})~,\qquad \qquad 
\check{\cal D}Z= {\cal D} Z ~, \qquad\quad Z\in\Omega^0(Y,{\cal Q})
~,\\
\check{\cal D}&: \Omega^1(Y,{\cal Q})\rightarrow\Omega^2_7(Y,{\cal Q})~,\qquad\qquad 
\check{\cal D}Z = \pi_7({\cal D}Z)
~, \quad Z\in\Omega^1(Y,{\cal Q})
~,\\
\check{\cal D} &: \Omega^2(Y,{\cal Q})\rightarrow\Omega^3_1(Y,{\cal Q})~,\qquad\qquad
\check{\cal D}Z = \pi_1({\cal D} Z)~, \quad Z\in\Omega^2(Y,{\cal Q})~.
\end{align*}
where the $\pi_i$'s denote projections onto the corresponding subspace.
\end{definition}
\noindent 
This definition extends the definition of the operators $\check d$ appropriate for deformations of manifolds with an integrable $G_2$ structure which where first considered in \cite{fernandez1998dolbeault, carrion1998generalization} and used in \cite{delaOssa:2016ivz}. 

The fact that $\check{\cal D}^2 = 0$ encodes the heterotic $G_2$ system, is stated in theorems 6 and 7 in our paper \cite{delaOssa:2017pqy}.  In this note, we put these results together in the following theorem.

\begin{theorem}
\label{tm:alpha}
Let $Y$ be a manifold with a $G_2$ structure $\varphi$,  $V$ a bundle on $Y$ with connection $A$, and $TY$ the tangent bundle of $Y$ with connection $\theta$. Let $\zeta$ be a connection one-form on $TY$ compatible with the $G_2$ structure, that is, $\zeta$ is related to a metric connection $\nabla$ with $\nabla\varphi = 0$ and such that its $\Gamma$ symbols are related to $\zeta$ by
\be
 \zeta_a{}^b = \Gamma_{ac}{}^b\, \dd x^c~.
 \label{eq:zeta}
 \ee
Consider the exterior covariant derivative $\check{\cal D}$ defined by  \eqref{eq:calD} and definition \ref{def:g2proj}. Then  $~\check{\cal D}^2= 0$ if and only if 
$\big( [Y,\varphi],[V,A],[TY,\theta],H \big)$ is a heterotic $G_2$ system
and we choose 
\be \Gamma_{ab}{}^c = \Gamma^{\scriptstyle{LC}}_{~ab}{}^{~ c}
+ \frac{1}{2}\, H_{ab}{}^c ,  
\label{eq:zetaconn}
\ee
where $H$ is given by equation \eqref{eq:Htor}.
This statement is true to all orders in the perturbative $\alpha'$ expansion.
\end{theorem}

We will not give the full proof of this theorem here. For more details the reader is referred to \cite{delaOssa:2017pqy}.
However several comments are in order.  It is an easy check to see that the heterotic $G_2$ system implies  $~\check{\cal D}^2= 0$. The converse, is a little more involved. To see this we compute $\check{\cal D}^2$, which gives
\be
\check{\cal D}^2 = 
\begin{pmatrix}
~\check\dd_{\zeta}^2 + \check{\cal R}^2 - \check{\cal F}^2
&  
~~\check{\dd_\zeta}\,\check{\cal R}+ \check{\cal R}\,\check{\dd_\theta} 
& 
~~- (\check{\dd_\zeta}\,\check{\cal F}+ \check{\cal F}\,\check\dd_A)~
 \\[3pt] 
 ~~\check{\cal R}\,\check{\dd_\zeta} + 
 \check{\dd_\theta}\, \check{\cal R}~
 & 
 ~~~\check\dd_{\theta}^2 + \check{\cal R}^2
 & 
 ~~- \check{\cal R}\,\check{\cal F}
 \\[3pt] 
 ~~\check{\cal F}\,\check{\dd_\zeta} + 
 \check\dd_A\, \check{\cal F}~ 
 & 
 ~~   \check{\cal F}\,\check{\cal R}
 & 
 ~~~\check\dd_A^2 - \check{\cal F}^2
\end{pmatrix}\label{eq:op2}
\ee
 The proof starts with the $(3,3)$ component of $\check{\cal D}^2$. One needs to assume that when the constant $\alpha'$ vanishes, one has the condition that $F_0\wedge\psi = 0$, where $F_0$ is the curvature of the one form connection $A$ on the bundle $V$ when $\alpha' = 0$.  This condition is motivated by the physics of heterotic string compactifications. Working order by order in the perturbative $\alpha'$  expansion, one finds that $F\wedge\psi=0$ to all orders in $\alpha'$. Moreover one finds that the $G_2$ structure must be integrable.  This means that the connection  one form $\zeta$ must be compatible with the integrable $G_2$ structure.  Similar considerations apply for the $(2,2)$ component, that is, $R\wedge\psi=0$.  With the exception of the component $(1,1)$, it is now straightforward to prove that all other components vanish.  Note that the vanishing of the components $(1,2)$, $(1,3)$, $(2,1)$, $(3,1)$ require also the Bianchi identities $\dd_A\, F = 0$ and $\dd_\theta\, R = 0$. Note that integrability also implies that $\nabla\varphi=0$, where $\nabla$ has connection symbols given by \eqref{eq:zetaconn} and $H$ is again given by \eqref{eq:Htor}. Finally the vanishing of the component $(1,1)$ then {\it implies} the Bianchi identity for $H$. 
  
 The necessity of the anomaly cancelation condition for $\check{\cal D}^2=0$ is  a rather striking result as this represents a new interpretation of the anomaly cancelation condition.  It means that the Bianchi identity of the anomaly, together with the conditions that $\varphi$ is an integrable $G_2$ structure on $Y$,  the connections $A$ on $V$ and $\theta$ on $TY$ are instantons, 
 are the necessary and sufficient condition for there to exist a differential complex $\check\Omega^*(Y, {\cal Q})$ 
\[
0\rightarrow\Omega^0(Y, {\cal Q})
\xrightarrow{\check{\cal D}}\Omega^1(Y, {\cal Q})
\xrightarrow{\check{\cal D}}\Omega_{\bf 7}^2(Y, {\cal Q})
\xrightarrow{\check{\cal D}}\Omega_{\bf 1}^3(Y, {\cal Q})
\rightarrow 0~,
\]
where $\Omega_{\bf r}^n(Y, {\cal Q})$ is the space of $\cal Q$ valued  $n$-forms in the $\bf r$ irreducible representation of $G_2$.
Moreover, it has also been shown (see \cite{delaOssa:2017pqy,Clarke:2016qtg}) that this complex is elliptic, which means that the $\check{\cal D}$ cohomology of $\cal Q$ is finite dimensional provided  that $Y$ is compact.  
  
The only part of the proof of the theorem that may not be rigorous mathematically has to do with the need to use a perturbative 
 $\alpha'$ expansion. This is motivated by  superstring theory. In this setting,  we recall that $\alpha'$ is related to the string length, which is conventionally assumed to be small so that strings appear point-like at low energies. With this assumption, superstring theory reduces to a supergravity theory with an infinite tower of $\alpha'$ corrections. The analysis we have performed is sensible under this premise.

Another interesting result in \cite{delaOssa:2016ivz, delaOssa:2017pqy} pertains to the infinitesimal moduli of heterotic $G_2$ systems.
Consider a family of heterotic systems $\big( [Y,\varphi],[V,A],[TY,\theta],H \big)_t $ parametrised by $t\in\IR$  with 
\[
\big( [Y,\varphi],[V,A],[TY,\theta],H \big)_{t=0} = \big( [Y,\varphi],[V,A],[TY,\theta],H \big)~.
\]
Let ${\cal Z}_t\in \Omega^1(Y,{\cal Q})$ with
\be
{\cal Z} = \begin{pmatrix}
M_t \\ \kappa_t\\ \alpha_t
\end{pmatrix}~.\label{eq:Zeta}
\ee
where $M_t\in\Omega^1(Y, T^*Y)$, $\alpha_t\in\Omega^1(Y, {\rm End}(V)$, and $\kappa_t\in\Omega^1(Y,{\rm End}(TY)$ represent infinitesimal deformations of the integrable $G_2$ structure, and the instanton connections $A$ and $\theta$ respectively.  Note that we describe the variations of $\varphi$ and $\psi$ in terms of $M_t$ and these are related by
\[
\partial_t\varphi = i_{M_t}(\varphi)~,\qquad
\partial_t\psi = i_{M_t}(\psi)~.
\]
Considering $M_t$ as a matrix, the antisymmetric part of $M_t$ in the
$\bf 14$ irreducible representation of $G_2$ drops out of these relations. 
Note however that while this is not part of the moduli of the $G_2$ structure it is needed to incorporate the (gauge invariant) variations of the $B$ field as discussed below. 

We prove  in \cite{delaOssa:2017pqy} that the tangent space to moduli space of heterotic $G_2$ systems is given
by the first  $\check{\cal D}$ cohomology of the bundle $\cal Q$
\[
({\cal TM})_{Het} = H^1_{\check{\cal D}}(Y, {\cal Q})~.
\]
Remarkably, we find that in order for this to work, we have to identify the symmetric part of $M_t$ with the variations of the metric on $Y$ and the gauge invariant variations ${\cal B}_t$ of the $B$ field with the antisymmetric part of $M$, where ${\cal B}_t$
is obtained  by varying equation \eqref{eq:anom} 
\[
\partial_t H = \dd {\cal B}_t + \frac{\alpha'}{2}\,
\tr\big( F\wedge\alpha_t - R\wedge\kappa_t\big)
~.
\]
An interesting fact is that in the case of a manifold with $G_2$ holonomy
we recover the well  known result \cite{joyce1996:1,joyce1996:2,Dai2003,deBoer:2005pt} that the dimension of the moduli space of $G_2$ metrics together with the $B$ field is given by $b_3 + b_2$ of the manifold $Y$.  

%%%%%%%%%%%%%%%%%%%%%%%%%%%

\section{Reducing the heterotic $G_2$ system}

In this section we consider the heterotic $G_2$ system upon restriction of the seven dimensional manifold $Y$ to six dimensions. We thus assume that our seven dimensional spacetime $Y$ is of the form
\begin{equation}
\label{eq:Yred}
Y=\mathbb{R}\times X\:,
\end{equation}
where now $X$ is some compact six dimensional manifold with an $SU(3)$ 
structure. The aim of the section is to recover the Strominger-Hull system  \cite{Strominger:1986uh,Hull:1986kz}
on $X$, which will be defined below. 

Under the restriction, the $G_2$ structure decomposes as
\begin{align}
\varphi&=\dd r\wedge\omega+{\rm Re}(\Psi)\\
\psi&=-\dd r\wedge{\rm Im}(\Psi)+\tfrac12\omega\wedge\omega\:.
\end{align}
where $r$ denotes the coordinate along the $\mathbb{R}$-direction. Here $(\Psi,\omega)$ denote the $SU(3)$ structure on $X$ satisfying the usual relations
\begin{equation} \label{eq:su3cond}
\omega\wedge\Psi=0\:,\;\;\;\;\tfrac{i}{\vert\vert\Psi\vert\vert^2}\Psi\wedge\bar\Psi=\tfrac16\omega\wedge\omega\wedge\omega\:.
\end{equation}
Six dimensional manifolds with an $SU(3)$ structure $(\omega,\Psi)$
are almost complex hermitian manifolds with an almost complex structure $J$ determined by the three form $\Psi$, and hermitian form
\[
\omega(U, V) = g(JU, V)~,
\]
for any vector fields $U$ and $V$. With respect to the almost complex structure $J$, the fundamental form $\omega$ is type (1,1) and $\Psi$ is type $(3,0)$.

We are interested in geometries of the form
\begin{equation}
M_{10}= M_{1,3}\times X\:,
\end{equation}
where $M_{1,3}$ is taken to be maximally symmetric. Furthermore, we assume no field dependence on the $r$ coordinate, and none of the six-dimensional form fields have  components along $\IR$ in \eqref{eq:Yred}. Under these assumptions, the heterotic $G_2$ system on $Y$ must restrict to a heterotic $SU(3)$ system on $X$ corresponding to the Strominger--Hull system \cite{Strominger:1986uh,Hull:1986kz}.  In particular, the assumptions combined with the integrability of the $G_2$ structure implies \cite{delaOssa:2014lma} that 
 $\tau_0= 0$ and hence that $M_{1,3}$ must be Minkowski space,  and moreover, that the $SU(3)$ structure on $X$ is constrained by  
 \begin{align}
\label{eq:complexX}
\dd\Psi&={\overline W}_1^{\, \Psi}\wedge\Psi~,
\\
\dd(\omega\wedge\omega)&=2\,W_1^\omega\wedge\omega\wedge\omega~,\label{eq:confbal}
\end{align} 
where
\begin{equation}
\label{eq:CompLee}
{\rm Re}(W_1^{\,\Psi})=W_1^\omega\:.
\end{equation} 
The restriction equation \eqref{eq:complexX} on the $SU(3)$ structure implies that $X$ is complex, that is, $J$ is integrable.  
 The second equation \eqref{eq:confbal} defines a locally conformally balanced metric on $X$ \footnote{When $W_1^\omega$ is $\dd$-exact, it defines a conformally balanced metric on $X$.}.  
For heterotic string compactifications, the Lee-form $W_1^\omega$ is also taken to be exact, and related to the dilaton by $W_1^\omega=\dd\phi$. As for the heterotic $G_2$ system, we shall forgo this assumption. 
It is also worth pointing out that, by applying $\bp$ to equation \eqref{eq:complexX}, ${\overline W}^{\, \Psi}_1$ is always $\bp$-closed, where $\bp$ is the anti-holomorphic Dolbeault operator associated to $J$.  The intrinsic torsion $H$ of the heterotic $G_2$ system in equation \eqref{eq:Htor} is reduced to
\be
H = - \dd^c \omega = i (\partial - \bar{\partial}) \omega \in \Omega^{1,2}(X) \oplus \Omega^{2,1}(X) \; .
\label{eq:HtorX}
\ee

Consider now the restriction to six dimensions of the vector bundles $V$ and $TY$.  The instanton conditions on the connections $A$ and $\theta$,
together with the assumption that the restrictions to $X$ are independent of the coordinate $r$, lead to
\begin{align}
\label{eq:HolYMA}
F \wedge \psi&=0 \implies F \wedge \Psi = 0~,\quad{\rm and}\quad
 F \w \omega \w \omega = 0 ~,\\
\label{eq:HolYMtheta}
R \wedge \psi&=0 \implies R \wedge \Psi = 0 
~,\quad{\rm and}\quad R \w \omega \w \omega= 0\:,
\end{align}
where on the right hand side $F$ and $R$ are  the curvatures of the connections $A$ and $\theta$, which are now, by a slight abuse of notation, restricted to the bundles $V|_X$ and  $TX$.
These are the Yang--Mills equations for holomorphic vector bundles for $V|_X$ and  $TX$.  In other words, the curvatures of $V|_X$ and  $TX$  are primitive type $(1,1)$ two forms. We will denote the antiholomorphic parts of the connections by $\vartheta = \theta^{(0,1)}$ and ${\cal A} = A^{(0,1)}$.  
Finally, by equations \eqref{eq:hetBI} and \eqref{eq:HtorX}
the heterotic Bianchi identity becomes
\be
\label{eq:SHBianchi}
-\dd \dd^c \omega=\dd H=\tfrac{\alpha'}{4}\left(\tr(F\wedge F)-\tr( R\wedge R)\right)\:.
\ee

For the sake of clarity, let us give a formal definition of the Strominger--Hull system:

\begin{definition}
\label{df:SH}
Let $X$ be a manifold with an $SU(3)$ structure $(\omega, \Psi)$, satisfying equations \eqref{eq:su3cond}, $V$ a bundle on $X$ with connection $A$, and $TX$ the tangent bundle of $X$ with connection $\theta$. A Strominger--Hull system $\big( [X,\omega,\Psi],[V,A],[TX,\theta],H \big)$  has a complex conformally balanced $SU(3)$ structure on $X$ and the curvatures $F$ and $R$ of the connections $A$ and $\theta$ respectively satisfy the holomorphic Yang--Mills conditions \eqref{eq:HolYMA}-\eqref{eq:HolYMtheta}, and the intrinsic torsion $H$ as defined by \eqref{eq:HtorX} satisfies the heterotic Bianchi identity \eqref{eq:SHBianchi}.
\end{definition}

\subsection{Reducing the operator $\cal D$}
Consider the operator $\cal D$ defined by equation \eqref{eq:calD}.
This differential operator can be restricted to a differential operator acting on 
\[
{\cal Q}\vert_X=T^*X\oplus\End(TX)\oplus\End(V)~,
\]
 where the individual bundles have been appropriately reduced. The restricted operator on $X$, which we denote ${\cal D}\vert_X$, further splits as
\begin{equation}
\label{eq:reducedD}
{\cal D}\vert_X=D+\overline D\:,
\end{equation}
where $D$ and $\overline D$ denote respectively the $(1,0)$ and $(0,1)$ part 
of ${\cal D}\vert_X$ with respect to the complex structure $J$  on $X$. We are particularly interested in the action of $\overline D$ on ${\cal Q}\vert_X$. Indeed, as we shall see, when splitting the cotangent bundle into its holomorphic and anti-holomorphic parts
\begin{equation*}
T^*X=T^{*(1,0)}X\oplus T^{*(0,1)}X\:,
\end{equation*}
by slight abuse of notation we can view $\overline D$ as an almost holomorphic structure acting on the topological sum of bundles
\begin{equation}
Q=T^{*(1,0)}X\oplus \End(TX)\oplus\End(V)\oplus T^{(1,0)}X\:. \label{eq:Q}
\end{equation}
In this section we see how this works.

The restriction of the operator $\cal D$ to $X$ can be deduced by studying the restriction on its action on sections on $Y$ with values in $\cal Q$. Indeed, the restricted operator gives a covariant derivative 
\begin{equation}
{\cal D}\vert_{X}=\dd+\Gamma_{X}\:,
\end{equation}
whose connection one forms $\Gamma_{X}$ are fully determined by their actions on sections. Moreover, to understand $\bar D$ we only need to restrict the action of $\cal D$ on a section ${\cal Z} \in\Omega^0(Y,\cal Q)$ to a corresponding action of $\overline D$ on sections with values in $Q$. Since we are interested in the restricted differential, we may assume that ${\cal Z}$ has no components in the $r$ direction. Let us also write 
\[
{\cal Z}= \begin{pmatrix}
M\\ \kappa \\ \alpha
\end{pmatrix}~.
\]
 where $M\in\Omega^0(X, T^*X)$, $\kappa\in\Omega^0(X,\End(TX))$ and $\alpha\in\Omega^0(X, \End(V))$. Then
\be
\check{\cal D}{\cal Z} = {\cal D}{\cal Z}
= \begin{pmatrix}
~\dd_\zeta M + {\cal R}(\kappa) - {\cal F}(\alpha)~\\
\dd_\theta \kappa + {\cal R}(M)\\
\dd_A \alpha - {\cal F}(M)
\end{pmatrix}~.
\label{eq:DZ}
\ee

Let $\{z^\mu,~\mu = 1, 2, 3\}$ be complex coordinates on $X$ and $\{z^{\bar\mu},~\bar\mu = 1, 2, 3\}$ their complex conjugates. The first component of equation \eqref{eq:DZ} is an element of $\Omega^1(Y, T^*Y)$.
The restriction to the anti-holomorphic part of $T^*X$ gives   
\begin{equation*}
\left({\overline D} {\cal Z}\right)_{\bar \mu}=
\left(\dd_\zeta^{(0,1)}M\right)_{\bar \mu}
+\tfrac{\alpha'}{4}\Big(
\tr\, (\kappa\, R_{\bar \mu\bar \nu}\dd z^{\bar \nu})-
\tr\, (\alpha\, F_{\bar \mu\bar \nu}\dd z^{\bar \nu})
\Big)\:.
\end{equation*}
As we have seen above, the restriction of the curvatures $F$ and $R$ to $V|_X$ and $TX$ are of type $(1,1)$. Hence the last two terms in the above expression vanish. We are therefore left with
\begin{equation*}
\left({\overline D}{\cal Z}\right)_{\bar \mu}=\bp M_{\bar \mu}-\dd z^{\bar \rho}{\Gamma_{\bar \rho\bar \mu}}^{\bar \sigma}M_{\bar \sigma}-\dd z^{\bar \rho}{\Gamma_{\bar \rho\bar \mu}}^{\sigma}M_{\sigma}=g_{\bar \mu \sigma}\bp(g^{\sigma\bar \rho}M_{\bar \rho})\:.
\end{equation*}
The last equality follows from reducing the seven dimensional $\Gamma$ symbols to six dimensions and using \eqref{eq:HtorX}. Raising with the six dimensional hermitian metric the anti-holomorphic index to a holomorphic tangent bundle index, we can instead view this result as 
\begin{equation}
g^{\bar \mu \nu}\left({\overline D}{\cal Z}\right)_{\bar \mu}=\bp M^\nu\:.
\end{equation}
Next, consider the action of $\overline D$ on the holomorphic part of the cotangent bundle
\begin{equation*}
\left({\overline D}{\cal Z}\right)_{\mu}=
\left(\dd_\zeta^{(0,1)}M\right)_{\mu}
+\tfrac{\alpha'}{4}\Big(
\tr\, (\kappa\,R_{\mu \bar \nu}\dd z^{\bar \nu})
-\tr\, (\alpha\, F_{\mu \bar \nu}\dd z^{\bar \nu})\Big)\:.
\end{equation*}
Note that in this case, the terms involving the curvatures of the bundles do not vanish. The first term in the above equation reads
\begin{equation*}
\left(\dd_\zeta^{(0,1)}M\right)_{\mu}=\bp M_{\mu}-\dd z^{\bar \nu}{\Gamma_{\bar \nu \mu}}^{\bar \rho}M_{\bar \rho}-\dd z^{\bar \nu}{\Gamma_{\bar \nu \mu}}^{\rho}M_{\rho}\:.
\end{equation*}
As it turns out, the connection symbols ${\Gamma_{\bar \nu \mu}}^{\rho}$ of the reduced connection vanish. 
Hence, we can write this as 
\begin{equation*}
\left(\dd_\zeta^{(0,1)}M\right)_{\mu}=\bp M_{\mu}+\dd z^{\bar \nu}H_{\bar \nu \mu \rho}M^{\rho}\:.
\end{equation*}
Thus, we have 
\begin{equation}
\left({\overline D}{\cal Z}\right)_{\mu}
=\bp M_{\mu}+\dd z^{\bar \nu}H_{\bar \nu \mu \rho}M^{\rho}
+\tfrac{\alpha'}{4}\,
\Big(\tr\, (\kappa\, R_{\mu \bar \nu}\dd z^{\bar \nu})
-\tr\, (\alpha\, F_{\mu \bar \nu}\dd z^{\bar \nu})\Big)\:.
\end{equation}
Finally, we also have the action of $\overline D$ on the bundle valued part of the bundle, that is, we need to consider the reduction of the second  and third components of equation \eqref{eq:DZ}. Let us denote the endomorphism valued indices of the bundle by $\{\alpha,\beta,..\}$. 
This gives
\begin{equation}
\left({\overline D}{\cal Z}\right)^\beta=
{\dd_A^{(0,1)}}\alpha^\beta
+ M^{\mu}\, F^\beta_{\mu\bar \nu}\dd z^{\bar \nu}=\bp_A\alpha^\beta
+ M^{\mu}\,F^\beta_{\mu\bar \nu}\dd z^{\bar \nu}\:,
\end{equation}
where we denote ${\dd_A^{(0,1)}}=\bp_{\cal A} = \bp + [{\cal A}, ]$, and ${\cal A}=A^{(0,1)}$. The action is similar for the $\End(TX)$-part of the bundle.

Thus, in summary we find that $\check{\cal D}{\cal Z} = {\cal D}{\cal Z}$ restrict to $\overline{D} Z$ with
\be
{Z} = \begin{pmatrix}
W_t \\ \kappa_t \\ \alpha_t \\ M_t
\end{pmatrix}
~,\label{eq:Z}
\ee
where, by slight abuse of notation, we now have $W_t\in\Omega^1(X, T^{*(1,0)}X)$, $\alpha_t\in\Omega^1(X, {\rm End}(V))$, $\kappa_t\in\Omega^1(X,{\rm End}(TX))$ and $M_t\in\Omega^1(X, T^{(1,0)}X)$. In other words, as claimed above, $Z$ is an element of the bundle $Q$ of \eqref{eq:Q}, and $\overline{D} $ defines an almost holomorphic structure on this bundle.

\subsection{The Strominger--Hull system and the operator $\overline{D}$}
We would like to make a few remarks on the above  restriction of the heterotic $G_2$ system $([Y,\varphi], [V, A], [TY, \theta], H)$ to six dimensions. We have shown that, upon restriction to $Y= \mathbb{R} \times X$ the differential operator $\check{\cal D} \to D + \overline{D}$,  and that $\overline{D}$ acts on a holomorphic bundle
\[
Q=T^{*(1,0)}X\oplus \End(TX)\oplus\End(V)\oplus T^{(1,0)}X\:,
\]
which is exactly the extension bundle constructed for the six-dimensional Strominger--Hull system in Refs \cite{delaOssa:2014cia,Anderson:2014xha}. More precisely,  $\overline{D}$ matches the upper triangular operator of the same name defined in \cite{delaOssa:2014cia}. As such, it defines a holomorphic structure on $Q$ and maps
\[
\overline{D}: \Omega^{(p,q)} (X, Q) \to \Omega^{(p,q+1)} (X,Q) \; .
\]

The six-dimensional extension bundle $Q$, which has, since ${\overline{D}}^2=0$, a flat connection, is closely related to the heterotic Courant algebroids defined in Refs \cite{Baraglia:2013wua,Garcia-Fernandez:2013gja}. This relation is discussed in some detail in Ref \cite{Anderson:2014xha}, which also provides a succinct introduction to the literature on heterotic and transitive Courant algebroids (see also \cite{bouwknegt-string-math}). In brief, the Bianchi identity of the heterotic anomaly cancellation condition defines a Courant algebroid built from the vector bundle $TX \oplus {\rm ad}P \oplus T^*X$, where we have introduced the direct sum bundle $P = V \oplus TX$, and this has been used \cite{2013JGP....73..150R,Baraglia:2013wua,Garcia-Fernandez:2013gja} to connect heterotic supergravity solutions to Hitchin's generalised geometry \cite{Hitchin:2004ut,g04}.  The heterotic Courant algebroid has also been used to determine infinitesimal deformations of to heterotic compactifications to four and three dimensions \cite{Garcia-Fernandez:2015hja, Clarke:2016qtg}. Here the holomorphic extension bundle $(Q,\bar D)$ forms a crucial part of what the authors denote a {\it holomorphic} Courant algebroid, which can be defined when the heterotic supersymmetry conditions in of the Strominger-Hull system are satisfied. It is interesting to ask whether a similar object can be defined for the heterotic $G_2$ system. 

It is also interesting to note that the extension bundle $Q$ on a six dimensional manifold $X$, which contains both the holomorphic tangent $T^{(1,0)}X$ and cotangent bundle $T^{*(1,0)}X$, embeds into a bundle ${\cal Q}$ on a real seven dimensional manifold $Y$, which only contains the cotangent bundle $T^*Y$.  This provides a perspective on heterotic $G_2$ systems that is not readily implied by a heterotic Courant algebroid like the one constructed in \cite{Clarke:2016qtg}.  While it should be possible to relate the two perspectives, and repackage ${\cal (Q, D)}$ into an algebroid structure, we can also turn the question around and explore what the real perspective implies for the six-dimensional Strominger--Hull system. This is the aim of the next section.

\section{A corollary of theorem 1}

We saw above how to recover the holomorphic differential operator defined in \cite{delaOssa:2014cia}. Let us now return to the viewpoint of viewing ${\cal D}\vert_X$ as a {\it real operator} on the bundle ${\cal Q}\vert_X=T^*X\oplus\End(TX)\oplus\End(V)$, where now
\begin{equation}
T^*X=T^{*(1,0)}X\oplus T^{*(0,1)}X\:.
\end{equation}
With this, we now consider the consequences of Theorem \ref{tm:alpha} when reducing to six dimensions. The restricted exterior derivative ${\cal D}\vert_X$ has a curvature
\begin{equation}
F_{{\cal D}\vert_X}={\cal D}\vert^2_X=D^2+D\overline D+\overline D D+\overline D^2\:.
\end{equation}
Furthermore, imposing that the operator $\cal{D}$ satisfies the instanton condition 
\be \nn \check{\cal D}^2=\psi \w  F_{\cal D}=0
\ee
 in seven dimensions implies that the operator ${\cal D}\vert_X$ corresponds to holomorphic Yang--Mills connection, that is
\begin{equation} \nn
F_{{\cal D}\vert_X} \w \Psi = 0 \; , \quad
F_{{\cal D}\vert_X} \w \bar{\Psi} = 0 \; , \quad
\omega\wedge\omega\wedge F_{{\cal D}\vert_X}=0
\; .
\end{equation}

From Theorem \ref{tm:alpha} we then derive the following corollary

\begin{corollary}
\label{cor:alpha}
Let $(X,\Psi,\omega)$ be a manifold with an $SU(3)$ structure, $V$ a bundle on $X$ with connection $A$, and $TX$ the tangent bundle of $X$ with connection $\theta$. Let $\zeta$ be a metric connection with connection symbols given by \eqref{eq:zetaconn}, where $H=-\,\dd^c\omega$. Then the heterotic $SU(3)$ system $\left([X,\Psi,\omega],[V,A],[TX,\theta],H\right)$ satisfies the Strominger-Hull system, that is equations \eqref{eq:complexX}-\eqref{eq:SHBianchi} if and only if the operator ${\cal D}\vert_X$ defined by equation \eqref{eq:reducedD} satisfies the holomorphic Yang-Mills equations.
\end{corollary}
\begin{proof} 
This corollary is a direct consequence of Theorem \ref{tm:alpha}, when reducing it to six dimensions as in \eqref{eq:Yred}, and requiring that the resulting four-dimensional spacetime is maximally symmetric (which implies that the solutions are four dimensional Minkowski).  We also assume no quantity depends on the coordinate $r$ and that the fields have no components in the $r$-direction.
\end{proof}

We can then conclude that this real perspective on $d$-dimensional heterotic systems translates the differential constraints on the system into the existence of an operator ${\cal D}_d$ acting on ${\cal Q}_d$-bundle valued forms. Furthermore, $({\cal Q}_d,{\cal D}_d)$ satisfies the same constraint as does the vector bundle $(V_d,A_d)$  that describes the gauge sector of the heterotic system in $d$ dimensions. 
In this note we have identified $({\cal Q}_6,{\cal D}_6) =({\cal Q}|_X,{\cal D}|_X)$ and $({\cal Q}_7,{\cal D}_7) =({\cal Q},{\cal D})$ for six and seven dimensional heterotic systems,  respectively.  It is tempting to speculate that $({\cal Q}_d,{\cal D}_d)$ with the same property exist  for heterotic systems of other dimensions.

Let us also speculate about another possible consequence of corollary \ref{cor:alpha}. The Donaldson-Uhlenbeck-Yau theorem \cite{donaldson1985anti, uhlenbeck1986existence} and Li--Yau theorem \cite{Yau87} state that holomorphic {\it polystable bundles of zero slope} always admit connections that solve the hermitian Yang--Mills equations. Thus, instead of solving the hermitian Yang--Mills equations explicitly, we may use algebraic geometry to check whether the bundle satisfies certain stability constraints. It has been speculated by Yau and others that an analogous property holds for the full heterotic system of equations \cite{yau2010metrics, garcia2016lectures}. If such a theorem can be proven, one would be able to prove existence of solutions of the Strominger--Hull system by checking algebraic conditions. Our observation that the Strominger-Hull system can be rephrased in terms of the holomorphic Yang-Mills connection ${\cal D}|_X$ on a bundle ${\cal Q}|_X$ could potentially help further explorations in this direction. 

The restriction of heterotic $G_2$ systems that we have presented here gives a new perspective on the moduli of the Strominger--Hull system, that we will explore in an upcoming publication. It would also be interesting to use the restriction techniques of this note to explore the connections between heterotic systems in seven and six dimensions, without the additional constraints on the $r$-dependence of the fields assumed in this note. This would provide insight into the geometry and moduli of other types of $SU(3)$ structure compactifications of the heterotic string, in the spirit of refs. \cite{Gray:2012md,delaOssa:2014lma}. We hope to come back to this in the future.

%%%%%%%%%%%%%%%%%%%%%%%%%
\section*{Acknowledgements}
XD would like to thank the organisers of the Hitchin 70 conference for the opportunity to talk about some of the work presented here.
XD would like to acknowledge  the organisers of the 2017 Villa de Leyva Summer School on Geometric, Algebraic and Topological Methods for Quantum Field Theory during which some of this work was written. 
All the authors thank Marc-Antoine Fiset and Mario Garcia-Fernandez for discussions.
The work of XD is supported in part by the EPSRC grant EP/J010790/1. ML's research is financed by the Swedish Research Council (VR) under grant number 2016-03873. The work of EES, made within the Labex Ilp (reference Anr-10-Labx-63), was supported by French state funds managed by the Agence nationale de la recherche, as part of the programme Investissements d'avenir under the reference Anr-11-Idex-0004-02.

%\appendix
%%%%%%%%%%%%%%%%%%%%%%%%%
%\section{Definitions}\label{app:def}

\bibliographystyle{JHEP}

\bibliography{bibliography}
\end{document}